\documentclass[a4paper]{amsart}
\usepackage{amsmath}
\usepackage{amsfonts}

\newtheorem{theorem}{Theorem}
\theoremstyle{plain}

\newtheorem{lemma}{Lemma}

\newtheorem{remark}{Remark}

\input{tcilatex}

\begin{document}
\title[Interior-inverse Sturm-Liouville problem]{Mochizuki-Trooshin type
theorem for Sturm-Liouville problem on time scales}
\author{\.{I}. Adalar}
\curraddr{Zara Veysel Dursun Colleges of Applied Sciences, Sivas Cumhuriyet
University Zara/Sivas, Turkey}
\email{iadalar@cumhuriyet.edu.tr}
\author{A. S. Ozkan}
\curraddr{Department of Mathematics, Faculty of Science, Sivas Cumhuriyet
University 58140 Sivas, Turkey}

\begin{abstract}
In this paper, we consider an interior inverse Sturm-Liouville problem on
time scale $\mathbb{T=}[0,a_{1}]\cup \lbrack a_{2},l]$ and give a
Mochizuki-Trooshin type theorem.
\end{abstract}

\maketitle

\section{\textbf{Introduction }}

Time scale theory was introduced by Hilger in order to unify continuous and
discrete analysis \cite{Hilger}. This approach has applied quickly to
various area in mathematics. Sturm-Liouville theory on time scales was
studied first by Erbe and Hilger \cite{Erbe} in 1993. Some important results
on the properties of eigenvalues and eigenfunctions of a Sturm-Liouville
problem on time scales were given in various publications (see e.g. \cite%
{Agarwal}- \cite{Amster2}, \cite{Davidson}-\cite{Davidson3}, \cite{Erbe2}, 
\cite{Guseinov}, \cite{Guseinov2}, \cite{Hilger}, \cite{Simon}, \cite%
{Huseinov}-\cite{Sun} and the references therein).

Inverse spectral problems consist in recovering the coefficients of an
operator from their spectral characteristics. The first results on inverse
theory of classical Sturm-Liouville operator were given by Ambarzumyan and
Borg \cite{Ambartsumian}, \cite{Borg}. Inverse Sturm-Liouville problems
which appear in mathematical physics, mechanics, electronics, geophysics and
other branches of natural sciences have been studied for about ninety years
(see \cite{Atkinson}, \cite{chadan} and \cite{G}).

Specially, the inverse problem for interior spectral data of the
differential operator consists in reconstruction of this operator from the
given eigenvalues and some information on eigenfunctions at an internal
point. This kind of problems for the Sturm Liouville operator were studied
firstly by Mochizuki and Trooshin \cite{troos}. Similar results to Mochizuki
and Trooshin have been studied in various paper until today \cite{ada}-\cite%
{troos2}.

Although the literature for inverse Sturm--Liouville problems on a
continuous interval is vast, there is only a few studies about this subject
on time scales \cite{ozkan}, \cite{ozz} and \cite{yur}. Such problems is
useful in many applied problems, for example in string theory, in dynamics
of population, in spatial networks problems etc.

In this paper, we consider\ an interior inverse Sturm--Liouville problem on
a time scale and give a Mochizuki and Trooshin-type theorem. We hope that
our results will contribute to the development of inverse spectral theory on
time scales and to obtain stronger results in some applied sciences.

For basic concepts of the time scale theory we refer to the textbooks \cite%
{Bohner} and \cite{Bohner2}.

In this paper we give an Mochizuki-Trooshin type theorem in \cite{troos} on
a time scale. Throughout this paper we assume that $\mathbb{T=[}0,a_{1}]\cup
\lbrack a_{2},l]$ is a bounded time scale for $a_{1}<a_{2}.$ Let us consider
the boundary value problem. We consider following boundary value problem $L$%
\ on $\mathbb{T=[}0,a_{1}]\cup \lbrack a_{2},l]$ 
\begin{eqnarray}
&&\text{\ }\left. \ell y:=-y^{\Delta \Delta }(t)+q(t)y^{\sigma }(t)=\lambda
y^{\sigma }(t)\text{, \ \ }t\in \mathbb{T}^{k^{2}}\right. \medskip \\
&&\text{ }\left. U(y):=y^{\Delta }(0)-hy(0)=0\right. \medskip \\
&&\text{ }\left. V(y):=y^{\Delta }(l)+Hy(l)=0\right. \medskip
\end{eqnarray}%
where $q(t)$ is real valued continuous function on $\mathbb{T}$, $h,H\in 
\mathbb{R}
$ and $\lambda $ is the spectral parameter.

Together with $L$, we consider a boundary value problem $\widetilde{L}=L(%
\widetilde{q}(t),h,H)$ of the same form but with different coefficients $%
\widetilde{q}(t).$ We assume that if a certain symbol $s$ denotes an object
related to $L$ , then $\widetilde{s}$ will denote an analogous object
related to $\widetilde{L}.$

A function $y,$ defined on $\mathbb{T},$ is called a solution of equation
(1) if $y\in C_{rd}^{2}(\mathbb{T)}$ and $y$ satisfies (1) for all $t\in 
\mathbb{T}$. The values of the $\lambda $ parameter%
\c{}
for which (1)-(3) has nonzero solutions are called eigenvalues, and the
corresponding nontrivial solutions are called eigenfunctions \cite{Agarwal}.
It is proven in \cite{Agarwal} that the problem (1)-(3) has countable many
eigenvalues which are real, simple and bounded below, and can be ordered as $%
-\infty <\lambda _{1}<\lambda _{2}<...<\lambda _{n}<...$.

\section{\textbf{Main Result}}

Prior to calculations,we need some preliminaries.

Let $\varphi (t,\lambda )$ be the solution of (1) under the initial
conditions 
\begin{equation}
\varphi (0,\lambda )=1\text{, }\varphi ^{\Delta }(0,\lambda )=h
\end{equation}%
The zeros of the function $\Delta (\lambda )=\varphi ^{\Delta }(l,\lambda
)+H\varphi (l,\lambda )$ coincide with the eigenvalues of the problem
(1)-(3). It is proven in \cite{Oz} that the functions $\varphi (t,\lambda )$%
, $\varphi ^{\Delta }(t,\lambda )$ and so $\Delta (\lambda )$ are entire on $%
\lambda .$

It is clear that $\varphi (t,\lambda )$ satisfies the following integral
equation on $\left( 0,a_{1}\right) $%
\begin{equation}
\varphi (t,\lambda )=\cos \sqrt{\lambda }t+\frac{h}{\sqrt{\lambda }}\sin 
\sqrt{\lambda }t+\frac{1}{\sqrt{\lambda }}\int\limits_{0}^{t}\sin \sqrt{%
\lambda }(t-\xi )q(\xi )\varphi (\xi )d\xi .
\end{equation}%
On the other hands, $\varphi ^{\Delta }(t,\lambda )$ is continuous at $a_{1}$%
, and so the relation\newline
\begin{equation}
\left. a\varphi ^{\prime }(a_{1}-0)=\varphi (a_{2})-\varphi (a_{1})\right.
\end{equation}%
holds, where $a:=a_{2}-a_{1}.$ From (5) and (6), we have the next lemma \cite%
{ozz}.

\begin{lemma}
The following asymptotic formula holds for $\left\vert \lambda \right\vert
\rightarrow \infty $; 
\begin{equation*}
\varphi (t,\lambda )=\left\{ 
\begin{array}{c}
\cos \sqrt{\lambda }t+\dfrac{h}{\sqrt{\lambda }}\sin \sqrt{\lambda }%
t+O\left( \dfrac{1}{\sqrt{\lambda }}\exp \left\vert \tau \right\vert
t\right) ,\text{ \ \ \ }t\in \mathbb{[}0,a_{1}] \\ 
a^{2}\lambda \sin \sqrt{\lambda }a_{1}\sin \sqrt{\lambda }(t-a_{2})+O\left( 
\sqrt{\lambda }\exp \left\vert \tau \right\vert (t-a_{2}+a_{1}\right) ),%
\text{ \ \ \ }t\in \lbrack a_{2},l]%
\end{array}%
\right.
\end{equation*}%
\newline
where $\tau :=\func{Im}\sqrt{\lambda }.$
\end{lemma}

The zeros of the function 
\begin{equation}
\Delta (\lambda )=\varphi ^{\Delta }(l,\lambda )+H\varphi (l,\lambda
)=\varphi ^{\prime }(l,\lambda )+H\varphi (l,\lambda )
\end{equation}%
coincide with the eigenvalues of the problem (1)-(3). From Lemma 1 we have
the following asimptotic relation.

\begin{equation}
\Delta (\lambda )=a^{2}\lambda ^{3/2}\sin \sqrt{\lambda }a_{1}\cos \sqrt{%
\lambda }(l-a_{2})+O\left( \lambda \exp \left\vert \tau \right\vert
(l-a_{2}+a_{1}\right) )
\end{equation}%
If we assume $l-a_{2}=a_{1}$, then we can prove by using well-known Rouche's
theorem that $\left\{ \lambda _{n}\right\} _{n\geq 1}$ satisfies the
following asymptotic formula for $n\rightarrow \infty $:%
\begin{equation}
\sqrt{\lambda _{n}}=\frac{(n-1)\pi }{2a_{1}}+O\left( \dfrac{1}{n}\right)
\end{equation}

\begin{lemma}
The system of functions $\left\{ \cos 2\sqrt{\lambda _{n}}t\right\}
_{n=1}^{\infty }$ is complete in $L_{2}(0,a_{1}).$
\end{lemma}

\begin{proof}
$\ $Let $\left\{ \mu _{n}\right\} _{n\geq 1}$ be the eigenvalues of the
classical Sturm-Liouville problem 
\begin{eqnarray*}
&&\left. -y^{\prime \prime }+q(t)y=\lambda y,\text{ }t\in (0,a_{1})\right. \\
&&\left. y^{\prime }(0)-hy(0)=y^{\prime }(a_{1})+Hy(a_{1})=0.\right.
\end{eqnarray*}%
Then the asymptotic equality $\sqrt{\mu _{n}}=\frac{(n-1)\pi }{a_{1}}%
+O\left( \dfrac{1}{n}\right) $ holds for sufficiently large $n.$ Taking care
of (11) we get%
\begin{equation}
\overset{\infty }{\underset{n=1}{\dsum }}\left\vert 2\sqrt{\lambda _{n}}-%
\sqrt{\mu _{n}}\right\vert ^{2}<\infty .
\end{equation}

It is known \cite{volk}, \cite{har} that $\left\{ \cos \sqrt{\mu _{n}}%
t\right\} _{n=1}^{\infty }$ is complete in $L_{2}(0,a_{1}).$ By using Lemma
3.1. in \cite{har}, the proof is completed.
\end{proof}

We state the main result of this article.

Let $\Lambda :=\left\{ \lambda _{n}\right\} _{n\geq 1}$ and $\widetilde{%
\Lambda }:=\left\{ \widetilde{\lambda }_{n}\right\} _{n\geq 1}$ be the
eigenvalues sets of $L$ and $\widetilde{L},$ $\varphi (t,\lambda _{n})$ and $%
\tilde{\varphi}(t,\lambda _{n})$ are eigenfunctions related to this
eigenvalues, respectively.

\begin{theorem}
If $\Lambda =\widetilde{\Lambda },$ $a_{1}+a_{2}=l$ and 
\begin{equation*}
\frac{\varphi ^{\Delta }(a_{1},\lambda _{n})}{\varphi (a_{1},\lambda _{n})}=%
\frac{\tilde{\varphi}^{\Delta }(a_{1},\lambda _{n})}{\tilde{\varphi}%
(a_{1},\lambda _{n})}
\end{equation*}%
then $q(t)=\widetilde{q}(t)$\ on $\mathbb{T}$.
\end{theorem}

\begin{remark}
Let $x=\left\{ 
\begin{array}{cc}
t, & t\in \lbrack a_{1},l] \\ 
t-a_{2}+a_{1}, & t\in \lbrack a_{2,}l]%
\end{array}%
.\right. $From (5) and (6), the problem (1)-(3) can be replaced by the
following boundary value problem 
\begin{eqnarray*}
&&\left. -y^{\prime \prime }+q_{1}(x)y=\lambda y,\text{ }x\in
(0,2a_{1})\right.  \\
&&\left. y^{\prime }(0)-hy(0)=y^{\prime }(2a_{1})+Hy(2a_{1})=0,\right.  \\
&&\left. y(a_{1}+0)-y(a_{1}-0)=ay^{\prime }(a_{1}-0),\right.  \\
&&\left. y^{\prime }(a_{1}+0)-y^{\prime }(a_{1}-0)=\frac{1}{a}(a^{2}\lambda
+b+1)y(a_{1}+0),\right. 
\end{eqnarray*}%
where $a:=a_{2}-a_{1}>0,$ $b:=-a^{2}q(a_{1})-1.$ Thus the eigenvalue problem
considered in Theorem 1 can be transformed into the Sturm-Liouville problem
with an interior discontinuity. The interior inverse problems for
Sturm-Liouville operators with various discontinuity conditions are
discussed in [37-45].
\end{remark}

\section{\textbf{Proof}}

Now we are ready to prove our main result.

\begin{proof}[Proof of Theorem 1.]
Consider the following equalities

\begin{equation}
-\varphi ^{\Delta \Delta }(t,\lambda )+q(t)\varphi ^{\sigma }(t,\lambda
)=\lambda \varphi ^{\sigma }(t,\lambda )
\end{equation}%
\begin{equation}
-\widetilde{\varphi }^{\Delta \Delta }(t,\lambda )+\widetilde{q}(t)%
\widetilde{\varphi }^{\sigma }(t,\lambda )=\lambda \widetilde{\varphi }%
^{\sigma }(t,\lambda ).
\end{equation}%
From (11) and (12) one can obtain%
\begin{equation}
\left[ \varphi (t,\lambda )\widetilde{\varphi }^{\Delta }(t,\lambda
)-\varphi ^{\Delta }(t,\lambda )\widetilde{\varphi }(t,\lambda )\right]
^{\Delta }=\left[ q(t)-\widetilde{q}(t)\right] \varphi ^{\sigma }(t,\lambda )%
\widetilde{\varphi }^{\sigma }(t,\lambda ).
\end{equation}%
By $\Delta $-integrating both sides of (13) on $\left[ 0,a_{1}\right] ,$ we
get%
\begin{equation*}
\left[ \varphi (t,\lambda )\widetilde{\varphi }^{\Delta }(t,\lambda
)-\varphi ^{\Delta }(t,\lambda )\widetilde{\varphi }(t,\lambda )\right]
_{0}^{a_{1}}=\int\limits_{0}^{a_{1}-0}\left[ q(t)-\widetilde{q}(t)\right]
\varphi ^{\sigma }(t,\lambda )\widetilde{\varphi }^{\sigma }(t,\lambda
)\Delta t
\end{equation*}%
Put%
\begin{equation*}
H(\lambda ):=\left[ \varphi (t,\lambda )\widetilde{\varphi }^{\Delta
}(t,\lambda )-\varphi ^{\Delta }(t,\lambda )\widetilde{\varphi }(t,\lambda )%
\right] _{0}^{a_{1}}.
\end{equation*}
From the assumption of theorem and initial conditions (4) we have $H(\lambda
_{n})=0$ for all $\lambda _{n}\in \Lambda $ and so $\chi (\lambda ):=\dfrac{%
H(\lambda )}{\Delta (\lambda )}$ is entire on $\lambda .$ Additionally, it
follows from (8) and Lemma 1 that 
\begin{equation}
\left\vert \chi (\lambda )\right\vert \leq C\left\vert \lambda \right\vert
^{-1/2}.
\end{equation}%
Thus $\chi (\lambda )=0$ for all $\lambda $. Hence, $H(\lambda )\equiv 0$.

From (13), the equality$\ \int\limits_{0}^{a_{1}}\left[ q(t)-\widetilde{q}(t)%
\right] \varphi (t,\lambda )\widetilde{\varphi }(t,\lambda )dt=0$ is valid
on the whole $\lambda $-plane. On the other hand, the following
representation holds on $[0,a_{1}]$ 
\begin{equation*}
\varphi (t,\lambda )\widetilde{\varphi }(t,\lambda )=\frac{1}{2}\left[
1+\cos 2\sqrt{\lambda }t+\int\limits_{0}^{t}V(t,\tau )\cos 2\sqrt{\lambda }%
\tau d\tau \right] ,
\end{equation*}%
where $V(t,x)$ is a continuous function which does not depend on $\lambda .$
This is used to get 
\begin{eqnarray}
&&\left. \frac{1}{2}\int\limits_{0}^{a_{1}}\left[ q(t)-\widetilde{q}(t)%
\right] \left( 1+\cos 2\sqrt{\lambda }t\right) dt+\right. \\
&&\left. \text{ \ \ \ \ \ \ \ \ }+\int\limits_{0}^{a_{1}}\left[ q(t)-%
\widetilde{q}(t)\right] \left[ \int\limits_{0}^{t}V(t,\tau )\cos 2\sqrt{%
\lambda }\tau d\tau \right] dt=0.\right.  \notag
\end{eqnarray}%
Taking into account (15) we conclude that%
\begin{equation}
\int\limits_{0}^{a_{1}}\cos 2\sqrt{\lambda }t\left[ q(t)-\widetilde{q}%
(t)+\int\limits_{t}^{a_{1}}\left[ q(\tau )-\widetilde{q}(\tau )\right]
V(t,\tau )d\tau \right] dt=0
\end{equation}%
Therefore, it follows from the completeness of the functions $\cos 2\sqrt{%
\lambda }t$ in Lemma 2 that 
\begin{equation}
q(t)-\widetilde{q}(t)+\int\limits_{t}^{a_{1}}\left[ q(\tau )-\widetilde{q}%
(\tau )\right] V(t,\tau )d\tau =0,\text{ }t\in \lbrack 0,a_{1})
\end{equation}%
Since the equation (17) is a homogenous Volterra integral equation,\ it has
only trivial solution. Thus $q(t)=$ $\widetilde{q}(t)$ on $\left[ 0,a_{1}%
\right] .$ To prove that $q(t)=$ $\widetilde{q}(t)$ on $\left[ a_{2},l\right]
$, we will consider the supplementary problem $L_{1}$ :%
\begin{eqnarray*}
&&\left. -y^{\nabla \nabla }+q_{1}(t)y^{\rho }=\lambda y^{\rho },\text{ }%
t\in \mathbb{T},\right. \\
&&\left. y^{\nabla }(0)-Hy(0)=y^{\nabla }(l)+hy(l)=0,\right.
\end{eqnarray*}%
where $q_{1}(t)=q(l-t).$\ By using chain rule in \cite{bart}, we have $%
\varphi _{1}(t,\lambda )=\varphi (l-t,\lambda )$ satisfies the equation $%
-\varphi _{1}^{\nabla \nabla }+q_{1}(t)\varphi _{1}^{\rho }=\lambda \varphi
_{1}^{\rho }$ and the initial conditions $\varphi _{1}(l,\lambda )=1,$ $%
\varphi _{1}^{\nabla }(l,\lambda )=-h$. Furthermore, the assumption $\dfrac{%
\varphi _{1}^{\nabla }(a_{2},\lambda _{n})}{\varphi _{1}(a_{2},\lambda _{n})}%
=\dfrac{\tilde{\varphi}_{1}^{\nabla }(a_{2},\lambda _{n})}{\tilde{\varphi}%
_{1}(a_{2},\lambda _{n})}$ holds.

If we repeat the above arguments then we replace equation (13) by%
\begin{equation}
\left[ \varphi _{1}(t,\lambda )\widetilde{\varphi }_{1}^{\nabla }(t,\lambda
)-\varphi _{1}^{\nabla }(t,\lambda )\widetilde{\varphi }_{1}(t,\lambda )%
\right] ^{\nabla }=\left[ q_{1}(t)-\widetilde{q}_{1}(t)\right] \varphi
_{1}^{\rho }(t,\lambda )\widetilde{\varphi }_{1}^{\rho }(t,\lambda ).
\end{equation}%
By integrating (in the sense of $\nabla $-integral) both sides of this
equality on $\left[ 0,a_{2}\right] ,$ we obtain%
\begin{eqnarray*}
\left[ \varphi _{1}(t,\lambda )\widetilde{\varphi }_{1}^{\nabla }(t,\lambda
)-\varphi _{1}^{\nabla }(t,\lambda )\widetilde{\varphi }_{1}(t,\lambda )%
\right] _{0}^{a_{2}} &=&\int\limits_{0}^{a_{2}}\left[ q_{1}(t)-\widetilde{q}%
_{1}(t)\right] \varphi _{1}^{\rho }(t,\lambda )\widetilde{\varphi }%
_{1}^{\rho }(t,\lambda )\nabla t \\
&=&\int\limits_{0}^{a_{1}}\left[ q_{1}(t)-\widetilde{q}_{1}(t)\right]
\varphi _{1}(t,\lambda )\widetilde{\varphi }_{1}(t,\lambda )dt \\
&&+\int\limits_{a_{1}}^{a_{2}}\left[ q_{1}(t)-\widetilde{q}_{1}(t)\right]
\varphi _{1}^{\rho }(t,\lambda )\widetilde{\varphi }_{1}^{\rho }(t,\lambda
)\nabla t
\end{eqnarray*}%
Since $q(a_{1})=\widetilde{q}(a_{1})$, then $q_{1}(a_{2})=\widetilde{q}%
_{1}(a_{2})$ and\newline
$\int\limits_{a_{1}}^{a_{2}}\left[ q_{1}(t)-\widetilde{q}_{1}(t)\right]
\varphi ^{\rho }(t,\lambda )\widetilde{\varphi }^{\rho }(t,\lambda )\nabla t=%
\left[ q_{1}(a_{2})-\widetilde{q}_{1}(a_{2})\right] \varphi ^{\rho
}(a_{2},\lambda )\widetilde{\varphi }^{\rho }(a_{2},\lambda )\left(
a_{2}-a_{1}\right) =0.$ \newline
Therefore%
\begin{equation}
\left. \varphi _{1}^{\nabla }(a_{2},\lambda )\widetilde{\varphi }%
_{1}(a_{2},\lambda )-\varphi _{1}(a_{2},\lambda )\widetilde{\varphi }%
_{1}^{\nabla }(a_{2},\lambda )=\int\limits_{0}^{a_{1}}\left[ q_{1}(t)-%
\widetilde{q}_{1}(t)\right] \varphi _{1}(t,\lambda )\widetilde{\varphi }%
_{1}(t,\lambda )dt.\right.
\end{equation}%
Let

\begin{equation}
\left. K(\lambda ):=\int\limits_{0}^{a_{1}}\left[ q_{1}(t)-\widetilde{q}%
_{1}(t)\right] \varphi _{1}(t,\lambda )\widetilde{\varphi }_{1}(t,\lambda
)dt.\right.
\end{equation}%
It is obvious that $K(\lambda _{n})=0$ for all $\lambda _{n}\in \Lambda $
and so $\omega (\lambda ):=\dfrac{K(\lambda )}{\Delta (\lambda )}$ is entire
on $\lambda .$ On the other hand, from asymptotics $\varphi _{1}^{\nabla
}(a_{2},\lambda )$\ and $\varphi _{1}(a_{2},\lambda )$, $K(\lambda
)=O(\lambda \exp 2\left\vert \tau \right\vert a_{1})$ and $\left\vert \Delta
(\lambda )\right\vert \geq C_{\gamma }\left\vert \lambda \right\vert
^{3/2}\exp 2\left\vert \tau \right\vert a_{1}$ for sufficiently large $%
\left\vert \lambda \right\vert .$ Thus 
\begin{equation}
\left\vert \omega (\lambda )\right\vert \leq C\left\vert \lambda \right\vert
^{-1/2}.
\end{equation}%
From Liouville's Theorem $\omega (\lambda )=0$ for all $\lambda $. Hence, $%
K(\lambda )\equiv 0$.

By integrating\ again both sides of the equality (18) on $\left(
0,a_{1}\right) $, we get%
\begin{equation}
\varphi _{1}^{\prime }(a_{1},\lambda )\widetilde{\varphi }_{1}(a_{1},\lambda
)=\varphi _{1}(a_{1},\lambda )\widetilde{\varphi }_{1}^{\prime
}(a_{1},\lambda )
\end{equation}%
Put $\psi (t,\lambda ):=\varphi _{1}(a_{1}-t,\lambda ).$ It is clear that $%
\psi (t,\lambda )$ is the solution of the following initial value problem 
\begin{eqnarray*}
&&\left. -y^{\prime \prime }+q_{1}(a_{1}-t)y=\lambda y,\text{ }t\in
(0,a_{1})\right. \\
&&\left. y(a_{1})=1,\text{ }y^{\prime }(a_{1})=-H\right.
\end{eqnarray*}%
It follows from (22) that%
\begin{equation}
\psi ^{\prime }(0,\lambda )\widetilde{\psi }(0,\lambda )=\psi (0,\lambda )%
\widetilde{\psi }^{\prime }(0,\lambda ).
\end{equation}%
Taking into account Theorem 1.4.7. in \cite{G} it is concluded that $%
q_{1}(t)=$ $\widetilde{q}_{1}(t)$ on $\left[ 0,a_{1}\right] ,$ that is $%
q(t)= $ $\widetilde{q}(t)$ on $\left[ 0,a_{2}\right] .$ This completes the
proof.
\end{proof}

\end{document}